\documentclass[12pt,draft]{amsart}

\usepackage{amsmath, amsfonts, amssymb}
\usepackage{amsthm}
\usepackage[usenames]{color}
\usepackage[T1]{fontenc}
\usepackage{amsmath}
\usepackage{verbatim}
\usepackage{upref}
\usepackage{amsfonts,amssymb}
\usepackage{enumerate}
\usepackage{xcolor}
\usepackage{ulem}

\newcommand{\spo}{\text{\rm Supp}}
\newcommand{\Z}{\mathbb{Z}}
\newcommand{\F}{F\langle X|G\rangle}
\newcommand{\fx}{f(x_1,\ldots,x_n)}
\newcommand{\spa}{\text{\rm span}}

\newtheorem{thm}{theorem}[section]
\newtheorem{theorem}[thm]{Theorem}
\newtheorem{proposition}[thm]{Proposition}
\newtheorem{lemma}[thm]{Lemma}

\newtheorem{remark}[thm]{Remark}
\newtheorem{definition}[thm]{Definition}

\newtheorem{problem}[thm]{Problem}
\newtheorem{example}[thm]{Example}

\begin{document}

	\title[Graded monomial identities on matrices]{Graded monomial identities and almost non-degenerate gradings on matrices}
	
	\author[L. Centrone]{Lucio Centrone}\thanks{L. Centrone was supported by FAPESP (No. 2018/02108-7) and by CNPq grant No. 308800/2018-4}\address{Dipartimento di Matematica, Universit\`a degli Studi di Bari, via \linebreak Orabona 4, 70125, Bari, Italy}\email{lucio.centrone@uniba.it} \address{IMECC, Universidade Estadual de
		Campinas, Rua S\'ergio Buarque de Holanda, 651, Cidade Universit\'aria ``Zeferino Vaz'', Distr. Bar\~ao Geraldo, Campinas, S\~ao Paulo, Brazil, CEP
		13083-859}\email{centrone@unicamp.br}
	
	\author[D. Diniz]{Diogo Diniz}\thanks{D. Diniz was supported by CNPq grants No. 406401/2016-0 and No. 303822/2016-3 and by FAPESP grant 2018/15627-2}\address{Unidade Acad\^emica de Matem\'atica e Estat\'istica,		Universidade Fe\-deral de Campina Grande, Apr\'igio Veloso, 785, 58429970, Campina Grande, PB, Brazil}\email{diogo@mat.ufcg.edu.br}
	
	\author[T. C. de Mello]{Thiago Castilho de Mello}\thanks{T. C. de Mello was supported by CNPq grant No. 461820/2014-5 and FAPESP grants No. 2018/15627-2 and 2018/23690-6}\address{Instituto de Ci\^encia e Tecnologia, Universidade Federal de S\~ao Paulo, Av. Cesare M. Giulio Lattes, 1201, 12247014,	S\~ao Jos\'e dos Campos, SP, Brazil}\email{tcmello@unifesp.br}
	
	\maketitle
	
	\begin{abstract}
		
	Let $F$ be a field of characteristic zero, $G$ be a group and $R$ be the algebra $M_n(F)$ with a $G$-grading. Bahturin and Drensky proved that if $R$ is an elementary and the neutral component is commutative then the graded identities of $R$ follow from three basic types of identities and monomial identities of length $\geq 2$ bounded by a function $f(n)$ of $n$. In this paper we prove the best upper bound is $f(n)=n$, more generally we prove that all the graded monomial identities of an elementary $G$-grading on $M_n(F)$ follow from those of degree at most $n$. We also study gradings which satisfy no monomial identities but the trivial ones, which we call almost non-degenerate gradings. The description of non-degenerate elementary gradings on matrix algebras is reduced to the description of non-degenerate elementary gradings on matrix algebras that have commutative neutral component. We provide necessary conditions so that the grading on $R$ is almost non-degenerate and we apply the results on monomial identities to describe all almost non-degenerate $\Z$-gradings on $M_n(F)$ for $n\leq 5$. 
		
	\end{abstract}

\section{Introduction}

A fundamental problem in the theory of algebras with polynomial identities is finding a set of generators (or basis), as a verbal ideal (or $T$-ideal), for the ideal of identities of any given algebra. For associative algebras over a field of characteristic zero Kemer (see \cite{K},\cite{K2}) proved that there always exists a finite set of generators but he did not give any algorithm to determine such a set. A finite basis for the ideal of identities is known for commutative algebras, whereas for matrix algebras over an infinite field $F$ it is known only for the algebra of $2\times 2$ matrices, see \cite{D}, \cite{R}. If $F$ is a field of positive characteristic $p>2$ we have a description of a finite basis of identities for $M_2(F)$ (see \cite{Ko2}). We note that some further partial results for $M_2(F)$ in the case of fields of characteristic 2 were obtained in \cite{D1} and \cite{Ko1} but it is still unknown if the ideal of identities of $M_2(F)$ is finitely generated or not in this case. Hence finding out explicit sets of generators of $T$-ideals is in general an interesting problem. In this frame some "weaker" identities arise such as the graded polynomial identities for graded algebras. This kind of identities started to be studied in the theory developed by Kemer and plays an important role in the study of ordinary polynomial identities. In particular the $T$-ideal of any associative algebra can be obtained manipulating (via the Grassmann envelope) the ideal of $\Z_2$-graded identities of a finite dimensional $\Z_2$-graded algebra. Moreover two algebras graded by the same group satisfying the same graded identities must satisfy the same ordinary identities. A fundamental result by Bahturin, Giambruno and Riley (see \cite{bgr1}) shows that if an algebra is graded by a finite group and the component associated to the neutral element of the group is a PI-algebra, then the whole algebra is also a PI-algebra. All of this justifies the interest in studying graded identities.

A grading on a matrix algebra is said to be elementary if the elementary matrices are homogeneous in the grading. Group gradings on matrix algebras are described in terms of elementary gradings and division gradings, i.e., gradings in which every non-zero homogeneous element is invertible. If the group is cyclic or torsion-free, then every grading on a matrix algebra is an elementary grading, see the paper \cite{BZ1} by Bahturin and Zaicev and the references therein. From now on every field is assumed to be of characteristic 0. A basis for the graded identities of $M_2(F)$ with the canonical $\Z_2$-grading, and related algebras, was determined by Di Vincenzo in \cite{DV}. Later Vasilovsky provided a basis for the graded identities of the algebra of $n\times n$ matrices with the canonical gradings by the groups $\Z_n$ and $\Z$, see \cite{V}, \cite{V2}. In \cite{BD} Bahturin and Drensky described a basis for the identities of the algebra $M_n(F)$ of $n\times n$ matrices with an elementary grading such that the neutral component is commutative. In particular, they proved that every graded identity follows from three types of identities and monomial identities of length (the usual degree of the monomial) bounded by a function $f(n)$ of $n$. In \cite[Theorem 3.1]{CM} a similar result is obtained for every subalgebra of $M_n(F)$ having a basis consisting of elementary matrices with the canonical (or Vasilovsky) $\mathbb Z_n$-grading. The authors proved that in the ideal of graded identities we only need monomial identities of degree up to $2n-1$. It was conjectured in \cite[Conjecture 5.1]{CM} that the best upper bound is $n$. In \cite{DM} the authors extend such result, and the corresponding conjecture, for arbitrary groups and for infinite fields of arbitrary characteristic. In this paper we prove the last conjecture for $M_n(F)$. More precisely we prove that the graded monomial identities of $M_n(F)$ with an elementary grading by a group $G$ follow from those of degree up to $n$, see Theorem \ref{monomialdegree} in the sequel.

The canonical $\mathbb{Z}$-grading on $M_n({F})$ has an interesting property: the monomial identities are consequence of the most \textit{trivial} ones, i.e., the indeterminates with homogeneous degrees that lie in the complement of the support of the grading.
A natural question is to consider the $G$-gradings on the matrix algebra $M_n(F)$ satisfying the same property, i.e., that each of its monomial identities follow from the trivial ones. We shall call such gradings  \textit{almost non-degenerate}. This notion is related to the so called {\textit{non-degenerate gradings}}, (see \cite[Definition 1.1]{AO}). Such gradings are defined as gradings $R=\oplus_{g\in G}R_g$ such that given an $n$-tuple $(g_1,\dots,g_n)$ of elements of $G$, we have $R_{g_1}\cdots R_{g_n}\neq0$. Using the terminology of this paper, to be non-degenerate means $R$ does not satisfy any graded multilinear monomial identity. As examples, the canonical (or Vasilovsky) $\mathbb{Z}_n$-grading on $M_n({F})$ is non-degenerate, whereas the canonical $\mathbb{Z}$-grading on $M_n(F)$ does not satisfy this property. Indeed it is clear from the definition that if a finite dimensional graded algebra graded by a group $G$ has a non-degenerate $G$-grading, then $G$ is finite. Of course the notions of non-degenerate grading and almost non-degenerate grading are strictly related.
Notice that if $A$ is an algebra with an almost non-degenerate grading, then the grading on $A$ is non-degenerate if and only if its support coincides with the grading group. Hence if an algebra with an almost non-degenerate fails to be non-degenerate then it admits a trivial monomial identiy.

A relation between the grading group $G$ and the exponent of a PI-algebra $W$ ($\mathrm{exp}(W)$) that admits a non-degenerate grading is given in \cite{AO}: there exists an abelian subgroup $U$ of $G$ with $[G:U]\leq \mathrm{exp}\,(W)^K$, where $K$ is a constant not depending neither on $W$ nor $G$. The graded simple finite-dimensional algebras with a non-degenerate grading are easy to describe: a graded simple algebra $A$ has a non-degenerate grading if and only if the grading is strong, i.e., $A_gA_h=A_{gh}$ for every $g,h\in G$, see \cite[Lemma 3.3]{AO}. This motivates the following problem.
\begin{problem}\label{problem}
	Given a group $G$ determine the $G$-graded simple finite-dimensional algebras which are almost non-degenerate.
\end{problem} 
The above problem is reduced, using Proposition \ref{distinct-almost}, to the problem of describing the elementary gradings on matrix algebras with commutative neutral component that are non-degenerate. The polynomial identities of such gradings on $M_n(F)$ are completely described in \cite{BD}. However even for such gradings and $G=\Z$ the problem above is not trivial. In the last section we use Theorem \ref{monomialdegree} to solve Problem \ref{problem} for the $\Z$-graded matrix algebras $M_n(F)$, where $n\leq 5$ (see Theorems \ref{n=4} and \ref{n=5} in the text) and $F$ is a field of characteristic 0.

\section{Preliminaries}

All fields we refer to are assumed to be of characteristic zero and all algebras we consider are associative and unitary.

Let $G$ be a group with neutral element $1_G$ and let $F$ be a field. If $A$ is an $F$-algebra, the decomposition \[\Gamma:A=\bigoplus_{g\in G}A_g\] of $A$ as a direct sum of subspaces indexed by the elements of the group $G$ is called a \textit{$G$-grading} on $A$ if $A_{g}A_{h}\subseteq A_{gh}$ for every $g,h\in G$. The subspaces $A_{g}$, $g\in G$, are called \textit{homogeneous components} of the grading. An algebra $A$ with a fixed $G$-grading is said to be $G$-graded.
If $0\neq a\in A_g$ we say that $a$ is \textit{homogeneous of degree $g$}  
and we write $\deg a = g$. We define the \textit{support of $\Gamma$} as \[\spo(\Gamma):=\{g\in G|A_g\neq0\}.\] 

An ideal $I$ of $A$ is \textit{graded} by the group $G$ if $I = \oplus_{g\in G}(I\cap A_g).$ A graded algebra $A$ is called \textit{graded simple} if $A^2\neq0$ and the only graded ideals of $A$ are the zero ideal and $A$. Moreover a graded algebra is a \textit{graded division algebra} if every non-zero homogeneous element is invertible.

In order to compare two different gradings we need the notions of isomorphic, weak isomorphic and equivalent gradings.

Let $\Gamma_1:A=\oplus_{g\in G}A_g$ and $\Gamma_2:B=\oplus_{h\in H}B_h$ be a $G$-grading on $A$ and an $H$-grading on $B$ respectively. Let $\varphi:A\rightarrow B$ be an isomorphism of algebras. We say $\varphi$  is an \textit{isomorphism} of the gradings $\Gamma_1$ and $\Gamma_2$ if $G=H$, and for each $g\in G$, $\varphi(A_g)=B_g$.
If there exists an isomorphism of groups $\alpha:G\rightarrow H$ such that $\varphi(A_{g})=B_{\alpha(g)}$ for all $g\in G$, then we say $\varphi$ is a \textit{weak isomorphism} of the gradings $\Gamma_1$ and $\Gamma_2$.
We say that $\varphi$ is an \textit{equivalence of the gradings} $\Gamma_1$ and $\Gamma_2$ if there exists a bijection $\alpha:\spo(\Gamma_1)\rightarrow \spo(\Gamma_2)$ such that $\varphi(A_{g})=B_{\alpha(g)}$ for every $g\in G$. 

Note that weakly isomorphic gradings are equivalent, but the converse, however, does not hold. For example, the $\mathbb{Z}$-gradings on $M_n(F)$ induced by $(0,1,\dots, n-1)$ and $(0,d,\dots, d(n-1))$, for $d>1$, are equivalent but are not weakly isomorphic.

Another important notion we are going to use is that of a coarsening of a grading.
Let $\Gamma:A=\oplus_{g\in G} A_g$ and $\Gamma^{\prime}:A=\oplus_{h\in H}A_h^{\prime}$ be a $G$-grading and an $H$-grading on $A$ respectively. We say $\Gamma^{\prime}$ is a \textit{coarsening} of $\Gamma$ if for every $g\in G$ there exists $h\in H$ such that $A_g\subseteq A_h^{\prime}$.

Let us list some useful examples of gradings.
Any algebra can be endowed with a \textit{trivial} $G$-grading, where $G$ is any group, if we set $A_{1_G}=A$ and for each $g\neq 1_G$ $A_g=0$. This is called \textit{trivial grading} on $A$. Let $E$ be the algebra generated by a countable infinite set $\{e_1, e_2, \dots\}$ with the condition $e_ie_j=-e_je_i$ for all $i,j\geq 1$. This algebra is called \textit{Grassmann algebra} and it can be seen its basis as a vector space is given by $\mathcal{B}=\mathcal{B}_0\cup\mathcal{B}_1$, where $\mathcal{B}_0=\{e_{i_1}\cdots e_{i_k}\,|\,i_1<\cdots <i_k, \,\text{$k$ is even}\}\cup \{1\}$ and $\mathcal{B}_1=\{e_{i_1}\cdots e_{i_k}\,|\,i_1<\cdots <i_k, \,\text{$k$ is odd}\}.$ We denote by $E_i$ the subspace of $E$ generated by $\mathcal{B}_i$. Then it is easily seen that $E$ is $\Z_2$-graded by $E=E_0\oplus E_1$ which is called \textit{canonical} $\Z_2$-grading of $E$.

A grading on a matrix algebra $R=M_n(F)$ is called \textit{elementary} if the elementary matrices $e_{ij}$ are homogeneous in the grading. Given an $n$-tuple $\overline{g}=(g_1,\ldots,g_n)$ of elements of $G$ we set $R_g=\spa\{e_{pq}\mid g_p^{-1}g_q=g\} $, then $R=\bigoplus_{g\in G}R_g$ is an elementary $G$-grading on $R$. Moreover, it was proved in \cite{Das} that every elementary grading on $M_n(F)$ can be defined this way. 
The importance of such gradings is that over an algebraically closed field of characteristic 0 every grading on $R$ by a finite group is obtained by a certain tensor product construction from an elementary grading and a fine grading (see \cite{BZ1}).

We will refer to the $\mathbb{Z}$-grading on $M_n(F)$ induced by $(0,1,\dots, n-1)$ as the \textit{canonical $\mathbb{Z}$-grading on $M_n(F)$}. Analogously the $\Z_n$-grading on $M_n(F)$ induced by $(\overline{0},\overline{1},\ldots,\overline{n-1})$ is called canonical $\Z_n$-grading on $M_n(F)$.

In order to give the analogous definition of a polynomial identity for graded algebras, we define a free object in the class of $G$-graded algebras.
Let $\{X^{g}\mid g \in G\}$ be a family of disjoint countable sets of variables. Set $X=\bigcup_{g\in G}X^{g}$ and denote by $\F$ the free associative algebra freely generated by the set $X$ over $F$. If $x\in X^g$, we set $\deg x=g$. For a monomial $m=x_{i_1}x_{i_2}\cdots x_{i_k}$ we set $\deg m=\deg x_{i_1} \cdot  \deg x_{i_2}\cdots\deg x_{i_k}$. Given $g \in G$ we denote by $\F_g$ the subspace of $\F$ spanned by the set $\{m=x_{i_1}x_{i_2}\cdots x_{i_k}\mid \deg m=g\}$. Notice that $\F_g\F_{h}\subseteq \F_{gh}$ for all $g,h \in G$. Thus \[\F=\bigoplus_{g\in G}\F_g\] is a $G$-graded algebra. We refer to the elements of $\F$ as \textit{graded polynomials}. 

An ideal $I$ of $\F$ is said to be a \textit{$T_{G}$-ideal} if it is invariant under all $F$-endomorphisms $\varphi:\F\rightarrow\F$ such that $\varphi\left(\F_g\right)\subseteq\F_g$ for all $g\in G$. A subset $S\subseteq I$ is a \textit{basis} for the $T_G$-ideal $I$ if $I$ is the intersection of all the $T_G$-ideals of $\F$ containg $S$. 

If $A$ is a $G$-graded algebra, a polynomial $\fx\in\F$ is said to be a \textit{graded polynomial identity} of $A$ if $f(a_1,a_2,\ldots,a_n)=0$ for all $a_1,a_2,\ldots,a_n\in\bigcup_{g\in G}A_g$ such that $a_k\in A_{\deg x_k}$, $k=1,\ldots,n$. If $A$ satisfies a non-trivial graded polynomial identity, $A$ is said to be a \textit{($G$-)graded PI-algebra}. We denote by $T_G(A)$ the set of all graded polynomial identities of $A$, it is clear that $T_G(A)$ is a $T_G$-ideal of $\F$. We remark that if the group $G$ is finite, then for every $G$-graded algebra $A$ the $T_G$-ideal $T_G(A)$ admits a finite basis, see \cite{AB}, \cite{I}. In the case the group $G$ is the trivial one we shall refer to ordinary polynomials (or simply polynomials), polynomial identities, $T$-ideals, etc.

Let $R$ be the matrix algebra $M_n(F)$ with an elementary grading induced by the $n$-tuple $\overline{g}=(g_1,\dots,g_n)\in G^{n}$. Of course, the neutral component $R_{1_G}$ is commutative if and only if the entries of $\overline{g}$ are pairwise distinct. In this case the following are graded polynomial identities for $R$.
\begin{eqnarray}
x_{1}x_{2}-x_{2}x_{1}=0, \deg x_1 =\deg x_2= 1_G \label{(3)}\\
x_{1}x_{2}x_{3}-x_{3}x_{2}x_{1}=0,  \deg x_1=\deg x_2^{-1} = \deg x_3\neq 1_G \\ 
x_{1}=0, R_{\deg x_1}=0 \label{(5)}.
\end{eqnarray}
In the identities of (2), we may consider $R_{\deg {x_3}}\neq0$  and  $R_{\deg {x_2}}\neq0$.

If $S$ is a subalgebra of $R$ admitting a linear basis consisting of elementary matrices, then $S=\oplus_{g\in G}S_g$, where $S_g=S\cap R_g$, is a $G$-grading on $S$.

In Theorem 3.7 of \cite{DM}, the authors prove the next result.

\begin{theorem}
Let $F$ be an infinite field and let $G$ be a group. Let us consider $M_n(F)$ with the elementary $G$-grading induced by the $n$-tuple $\overline{g}=(g_1,\dots, g_n)\in G^{n}$, where the elements $g_1,\dots, g_n$ are pairwise different. If $S$ is a subalgebra of $M_n(F)$ generated by elementary matrices $e_{ij}$, then a basis of the graded polynomial identities of $S$ consists of $(\ref{(3)})-(\ref{(5)})$ and a finite number of graded monomial identities of degree $p$ where $2\leq p\leq 2n-1$.
\end{theorem}

Determining a finite set of monomial identities as in the previous theorem is not a trivial task and some examples can be found in \cite{BD}, \cite{CM}. In \cite{BD} the authors proved the graded monomial identities of $M_n(F)$ are bounded by a function $f(n)$ whereas in \cite{CM} and in \cite{DM} it is conjectured that every graded monomial identity of the subalgebra $S$ as in the previous theorem follows from the graded monomial identities of degree at most $n$. In the next section we are going to prove such conjecture for $S=M_n(F)$. We start off with the following definition.

\begin{definition}\cite[Definition 17]{DN}
Let $R$ be a $G$-graded algebra, let
$I_0$ be the $T_G$-ideal generated by the set $\{x_1\mid \deg(x_1)\notin \spo(R)\}$. A monomial $M\in T_G(R)$ is a \textit{trivial graded monomial identity} for $R$ if $M\in I_0$, otherwise we say that $M$ is a non-trivial graded monomial identity for $R$.
\end{definition}

We recall that the $G$-grading $\Gamma:R=\oplus_{g\in G} R_g$ on the algebra $R$ is called \textit{non-degenerate} if for every tuple $(g_1,\dots, g_r)\in G^{r}$ we have $R_{g_1}\cdots R_{g_r}\neq 0$, see \cite[Definition 1.1]{AO}. Equivalently the grading $\Gamma$ is non-degenerate if it admits no monomial identity and in this case $\spo(\Gamma)=G$. The grading $\Gamma$ is called \textit{strong} if $R_gR_h=R_{gh}$. It is clear that if $\Gamma$ is strong and $\spo(\Gamma)=G$, then it is also non-degenerate. The canonical $\mathbb{Z}_n$-grading on $M_n(F)$ is a strong grading and, in particular, it is also a non-degenerate grading. The canonical $\mathbb{Z}$-grading on $M_n(F)$ is not a non-degenerate grading because it satisfies the identities $x_1$, where $|\deg x_1|\geq n$. This grading, however, satisfies no non-trivial graded monomial identities (see \cite{V} for the description of its graded identities). With this example in mind, we introduce here the concept of almost non-degenerate grading.

\begin{definition}
The $G$-grading $R=\oplus_{g\in G} R_g$ on the algebra $R$ is \textit{almost non-degenerate} if it satisfies no non-trivial multilinear graded monomial identities.
\end{definition}

In Section \ref{4} we consider the problem of determining the almost non-degenerate gradings on matrix algebras. A more specific problem is to determine the almost non-degenerate $\Z$-gradings on $M_n(F)$. Even this is a non-trivial task, and we provide necessary conditions so that certain gradings on $M_n(F)$ are almost non-degenerate and we completely solve this problem for $n\leq 5$.

Let $n$ be a positive integer $\overline{g}=(g_1,\dots, g_n)$ be a tuple of elements of the group $G$ and let $D=\oplus_{g\in G}D_g$ be a $G$-graded algebra. Then it is easy to verify that the algebra $M_n(F)\otimes D$ may be endowed with a $G$-grading such that the element $e_{ij}\otimes d$ is homogeneous of degree $g_i^{-1}hg_j$ for every $1\leq i,j \leq n$ and every $d\in D_h$. We refer to this as the grading on $M_n(F)\otimes D$ induced by $\overline{g}$. If $D=\oplus_{g\in G}$ is a division grading then $M_n(F)\otimes D$ is a graded simple algebra. The next result implies every finite-dimensional graded simple algebra is obtained in this way, it is a direct consequence of \cite[Corollary 2.12]{EK}.

\begin{theorem}\label{gsimple}
Let $G$ be a group and let $R$ be a $G$-graded algebra. If $R$ is a finite-dimensional graded simple algebra, then there exist a positive integer $n$, an $n$-tuple ${\overline{g}}$ of elements of $G$ and a graded division algebra $D$ such that $R$ is isomorphic to $M_n(F)\otimes D$ with the grading induced by $\overline{g}$.
\end{theorem}

\section{Monomial Identities in Matrix Algebras}
In this section we prove that all multilinear graded monomial identities of the full matrix algebra of size $n$ follow from those of degree $n$ provided the grading is elementary. 

First, we mention that for elementary gradings a monomial identity is a consequence of a suitable monomial identity.

\begin{remark}\label{multrem}
A monomial $M=x_{i_1}\cdots x_{i_m}$ is a graded monomial identity for a n elementary $G$-grading $A$ on a matrix algebra if and only if the multilinear monomial $\tilde{M}=x_1\cdots x_m$, where $\mathrm{deg}_G\, x_j=\mathrm{deg}_G\, x_{i_j}$, is a graded identity for $A$. Hence an elementary $G$-grading on a matrix algebra is almost non-degenerate if and only if it satisfies no non-trivial monomial identity.
\end{remark}

This is not true in general. For instance, let us consider the Grassmann algebra $E$ with its canonical $\Z_2$-grading. Of course $E$ satisfies the monomial identity $z^2$, where $z$ is an indeterminate of degree $\overline 1$, whereas it does not satisfy any multilinear monomial identity. 

Let us show now that when dealing with monomial identities in elementary graded matrix algebras induced by an n-tuple $(g_1,\dots, g_n)$ of elments of $G$, it is not a restriction for our purposes to consider $g_1,\dots, g_n$ to be pairwise distinct.

\begin{proposition}\label{distinct-almost}
	Let $A=B\otimes D$ be a $G$-graded matrix algebra, where $B=M_n(F)$ with the elementary grading induced by $(g_1,\dots, g_n)$ and let $D$ be a $G$-graded division algebra with support $H$. Let $h_1,\dots, h_k$ be pairwise different elements of the group $G$ such that $\{g_1H,\dots, g_nH\}=\{h_1H,\dots, h_kH\}$. Let $A_0=M_k(F)\otimes D$, be the grading induced by $(h_1,\dots, h_k)$. Then $M=x_{1}\dots x_{m}$ is a graded monomial identity for $A$ if and only if it is a graded monomial identity for $A_0$. In particular, the algebra $A$ has an almost non-degenerate grading if and only if $A_0$ has an almost non-degenerate grading.
\end{proposition}

\begin{proof}
Note that $\mathrm{supp}\, A=\mathrm{supp}\, A_0$. The result now follows if we prove that a multilinear monomial is an identity for $A$ if and only if it is an identity for $A_0$. Let $M=x_1\cdots x_m$ be a monomial in the free $G$-graded algebra. Note that $A_0$ is isomorphic to a suitable subalgebra of $A$. Hence if $M$ is a graded identity for $A$, then $M$ is a graded identity for $A_0$. Now we assume that $M$ is a graded identity for $A_0$. We may assume up to isomorphism that the entries of the tuple $(g_1,\dots, g_n)$ lie in the set $\{h_1,\dots,h_k\}$ and that the first $q_1$ entries are $h_1$, the next $q_2$ entries are $h_2$ and so on. Let $d_i=q_1+\cdots +q_i$ for $i=1,\dots, k$. Let $J_1=\{1,\dots d_1\}$ and $J_i=\{d_{i-1}+1,\dots,d_i\}$ for $i=2,\dots, k$. Note that $\cup_l J_l$ is a partition of $\{1,\dots, n\}$. Then $g_i=h_l$ if $i\in J_l$. For each $i$ in $\{1,\dots,n\}$ there exists a unique $l$ such that $i\in J_l$ and we denote by $i^{\prime}$ the smallest integer in $J_l$. Now let $(E_{i_1,j_1}\otimes d_1,\dots, E_{i_m,j_m}\otimes d_m)$ be an admissible substitution on $M$ by elements of $A$ and let $r$ be the result of this substitution, then 
\begin{align*}
r=(E_{i_1,j_1}\otimes d_1)\cdot\dots\cdot (E_{i_m,j_m}\otimes d_m)\\=E_{i_1j_1}\cdots E_{i_mj_m}\otimes d_1\cdots d_m.
\end{align*}
Now let 
\begin{align*}
r^{\prime}:=(E_{i_1^{\prime},j_1^{\prime}}\otimes d_1)\cdot\dots\cdot (E_{i_m^{\prime},j_m^{\prime}}\otimes d_m)\\=E_{i_1^{\prime}j_1^{\prime}}\cdots E_{i_m^{\prime}j_m^{\prime}}\otimes d_1\cdots d_m.
\end{align*}
If $r\neq 0$, then $r^{\prime}\neq 0$. On the other hand $r^{\prime}$ is the result of an admissible substitution for $M$ by elements of $A^{\prime}$, where $A^{\prime}=B^{\prime}\otimes D$ and $B^{\prime}$ is the subalgebra of $B$ generated by the elementary matrices $E_{i,j}$, where $i,j\in \{1^{\prime},\dots, n^{\prime}\}$. As a consequence $M$ is not a graded identity for $A^{\prime}$. However $A^{\prime}$ is isomorphic to $A_0$ and that is a contradiction. Thus $r=0$ and $M$ is a graded identity for $A$.
\end{proof}

Until the end of this section $R$ will denote the algebra $M_n(F)$ with the elementary grading induced by an $n$-tuple $\overline{g}=(g_1,\ldots,g_n)$ of pairwise distinct elements of $G$. For every $g\in \spo(R)$ let $\{e_{i_1j_1},\dots, e_{i_l,j_l}\}$ be the basis of $R_{g}$ consisting of elementary matrices. Let $I_g=\{i_1,\dots, i_l\}$, $J_g=\{j_1,\dots, j_l\}$. Since the entries of $\overline{g}$ are pairwise distinct given $i\in I_g$ there exists a unique $j\in J_g$ such that $e_{ij}\in R_g$. This defines an injective map
$\hat{g}:I_{g}\rightarrow J_g$, this a map was introduced in \cite{DM}.

We start off with the next result about graded multilinear monomial identities, which is a key step in the proof of the main theorem of this section.

\begin{lemma}\label{deg=e}
	Let $G$ be a group and let $R$ be the elementary $G$-grading on $M_n(F)$ induced by the $n$-tuple $(g_1,\dots,g_n)$ of pairwise distinct elements of $G$. For $h\in \spo(R)$, let us denote \[\displaystyle M_{h}:=\sum_{g_k^{-1}g_l=h}{e_{kl}}.\] Let $h_1,\dots, h_k$ be elements of $\spo(R)$. If $h_1\cdots h_k=1_G$, then the number of non-zero lines in $M_1=M_{h_1} M_{h_2}\cdots M_{h_k}$ and $M_2= M_{h_2}\cdots M_{h_k}$ is the same.
\end{lemma}
	
	\begin{proof}
		Let $i_1,\dots,i_r$ be the non-zero lines of $M_1$ which are, of course, also the non-zero lines of $M_{h_1}$. Then $j_1=\widehat{h}_1(i_1)$, \dots, $j_r=\widehat{h}_1(i_r)$ are non-zero lines of $M_2$. The claim follows once we  prove $j_1,\ldots,j_r$ are exactly the non-zero lines of $M_2$.
		
		Since $M_{h_1}$ is homogeneous of degree $h_1$, we have $M_2$ is homogeneous of degree $h_1^{-1}$. If the $j$-th line of $M_2$ is non-zero, then there exists $i$ such that $e_{ji}$ has degree $h_1^{-1}$.
		Of course $\deg(e_{ij})=h_1$ and $j=\widehat{h}_1(i)$, where $i$ is a non-zero line of $M_1$, since $e_{ij}M_2\neq 0$ and the result follows.
	\end{proof}

As an easy consequence of the previous result we get the next result.

	\begin{lemma}
		Let $m=x_{1}\cdots x_{k}$ be a graded monomial identity of homogeneous degree $1_G$ of $R$. Then it is a consequence of the graded monomial identity $m'=x_{2}\cdots x_{k}$.
	\end{lemma}
	
	Now we are ready for the main result of the section.

	\begin{theorem}\label{monomialdegree}
		Let $G$ be a group and let $R$ be the algebra $M_n(F)$ endowed with a $G$-grading.  If $m=x_{1}\cdots x_{k}$ is a graded monomial identity for $R$ and $k>n$, then $m$ is a consequence of a graded monomial identity of $R$ of degree at most $n$.
	\end{theorem}
	
	\begin{proof}
	    {From Proposition \ref{distinct-almost} we may consider $R$ an elementary grading induced by the tuple $(g_1,\dots,g_n)$ of pairwise distinct elements of $G$.}
		Suppose $k>n$. If $x_{1}\cdots x_{n}$ is a graded monomial identity for $R$ we are done. Assume that $x_{1}\cdots x_{n}$ is not a graded identity for $R$. In this case there are indexes $i_1,j_1,\dots,i_n,j_n$ such that $\deg e_{i_r,j_r}=h_r$ for all $r\in\{1,\dots, n\}$ and $e_{i_1,j_1}\cdots e_{i_n,j_n}\neq0$. Of course, $j_r=i_{r+1}$, for all $r<n$. Defining $i_{n+1}:=j_n$, since $i_r\in\{1,\dots,n\}$, for all $r$, at least two among the indexes $i_1,\dots,i_{n+1}$ are equal. Let $i_s$ and $i_{t+1}=j_t$ be such indexes. Then $\deg x_{i_s}\cdots x_{i_t}=1_G$, i.e., $x_{1}\cdots x_{k}$ has a submonomial $m'=x_{i_s}\cdots x_{i_t}$ of degree $1_G$.
		
		Suppose first $s=1$, i.e., $m'$ is at the beginning of the monomial $m$. Then Lemma \ref{deg=e} shows that $m$ is a consequence of $x_{2}\cdots x_{k}$.
		
		Suppose now $s>1$. We claim the monomial \[x_{1},\dots,x_{s-2}yx_{s+1}\dots x_{k},\]  is a graded monomial identity for $R$, where $\deg x_i=h_i\in G$, for all $i$ and $\deg y=h_{s-1}h_{s}$. For this, it is enough to show that the non-zero lines of $M_{h_{s-1}}M_{h_s}\cdots M_{h_t}$ and $M_{h_{s-1}h_s}M_{h_{s+1}}\cdots M_{h_t}$ are the same. To prove this claim notice that any non-zero line of the former product is a non-zero line of the latter. 
		
		Now let $i$ be a non-zero line of $M_{h_{s-1}h_s}M_{h_{s+1}}\cdots M_{h_t}$. As before, there are matrices $e_{i_r,j_r}$ for $r\in \{s+1,\dots,t\}$ such that $\deg e_{i_r,j_r}=h_r$ and $e_{i,j}e_{i_{s+1},j_{s+1}}\cdots e_{i_{t},j_{t}}\neq 0$, where $\deg e_{i,j}=h_{s-1}h_s$. Since $h_s\cdots h_t=1_G$, we have $\deg e_{j_t,i_{s+1}}=h_s$. 
		
		Comparing degrees, one gets $\deg e_{i,j_t}=h_{s-1}$. Hence, \[e_{i,j_t}e_{j_t,i_{s+1}}e_{i_{s+1},j_{s+1}}\cdots e_{i_{t},j_{t}}\neq 0,\] which means $i$ is a non-zero line of $M_{h_{s-1}}M_{h_s}M_{h_{s+1}}\cdots M_{h_t}$.
		
		In order to complete the proof we have to argue by induction on $k\geq n+1$. If $k=n+1$ the discussion above shows $m$ is a consequence of a graded monomial of degree $n$ and we are done. Suppose the result is true for $k-1\geq n+1$ and let us prove it for $k\geq n+2$. As above, $m$ is a consequence of a graded identity of degree strictly less than $k$ that is, by induction, a consequence of a graded monomial of degree less than or equal to $n$ and now the proof is complete.
	\end{proof}

We remark that the proof of the above result uses the hypothesis that $R$ is a grading on $M_n(F)$, hence for an arbitrary subalgebra of $M_n(F)$ generated by elementary matrices the conjecture in \cite{DM} is still open.

{The next example shows that the bound $f(n)=n$ is sharp.} 

\begin{example}\label{exam}
 Let us consider $R=M_n(F)$ with the $\mathbb{Z}_{n+1}$-grading given by $(\overline 0,\overline 1,\dots, \overline{n-1})$. In this case, $R_{\overline 1}^{n}=0$ provides a graded monomial identity of length $n$ which does not follow from a graded monomial identity of lower degree. In fact, $R$ does not satisfy any monomial identity of degree less than $n$, indeed notice that if $g\neq 1_G$ then the component of $R_g$ is $n-1$-dimensional, so if we consider the graded monomial $x_{g_1}\dots x_{g_k}$, we get the linear span of its image (considered as a function of $R^k$ to $R$) is a vector subspace of $R_{g_1\cdots g_k}$ of dimension at least $n-k$.
\end{example}

In the next section we consider almost non-degenerate gradings on matrix algebras.
       
\section{Almost non-degenerate gradings}\label{4}
	
Let $G$ be a group. In this section we consider the problem of determining the $n$-tuples of elements of $G$ inducing almost non-degenerate elementary gradings on $M_n(F)$. We recall that almost non-degenerate gradings are gradings which satisfy no non-trivial graded monomial identities. 

Throughout this section $R$ will denote the matrix algebra $M_n(F)$ over an algebraically closed field of characteristic zero. Moreover we will use the word \textit{length} instead of \textit{degree}  of a monomial.
	
\
	
In the next paragraphs we would like to show why it is sufficient to consider the hypothesis of the neutral component being commutative, starting with we the following characterization result.

Due to the fact the ground field is algebraically closed, then $D_{1_G}=F$ is commutative and the neutral component of the algebra $A_0$ above is commutative which makes the hypothesis of the neutral component being commutative consistent at the light of the characterization of the graded algebras given by Bahturin and Zaicev in \cite{BZ1}.

Now we start the study of the almost non-degenerate gradings on $M_n(F)$.

	By Lemma \ref{distinct-almost}, we can restrict our attention to elementary gradings on $M_n(F)$ generated by $n$-tuples of pairwise distinct elements.

    \begin{remark}\label{isom}
        The elementary gradings on $M_n(F)$ induced by the tuples $(g_1,\dots, g_n)$ and $(h_1,\dots, h_n)$ in $G^{n}$ are isomorphic if and only if there exists a permutation $\pi \in S_n$ and $g\in G$ such that $h_i=g_{\pi(i)}g$, for $i=1,\dots, n$ (see \cite[Corollary 2.12]{EK}).
    \end{remark}
    
	As we have already mentioned, a direct consequence of the main result of \cite{V} is the following.
	
	\begin{proposition}\label{almnd}
	   The canonical $\mathbb{Z}$-grading on $M_n(F)$ is an almost non-degenerate grading.
	\end{proposition}
	
		In this section we determine sufficient conditions so that a $G$-grading on $M_n(F)$, where $G$ is a linearly ordered abelian group, turns out to be almost non-degenerate. We recall a \textit{linearly ordered group} is a group equipped with a total order that is translation-invariant. In particular, linearly ordered groups are always torsion-free. Since we are working with abelian groups, we will use additive notation for the group $G$. In particular, the neutral element will be denoted by $0$.
	
   The proof of the next lemma is easy and will be omitted.
   
        \begin{lemma}\label{campinense}
            Let $G$ be a linearly ordered abelian group, let $n$ be a positive integer and let $0=g_1<g_2<\dots<g_n\in G$ such that $g_j-g_i\in \{g_1,\dots, g_n\}$, for each $i$ and $j\in \{1,\dots, n\}$ with $i<j$. Then for all $k\in \{1,\dots, n\}$ we have $g_k=(k-1)g_2$.
        \end{lemma}

    \begin{remark}\label{grad}
        Any grading by a torsion-free group on a matrix algebra is an elementary grading (see \cite[Corollary 2.22]{EK}). In particular a grading by a linearly ordered group $G$ on a matrix algebra is an elementary grading. 
    \end{remark}

    The next proposition gives a characterization of $G$-gradings on matrix algebras that are equivalent to the canonical $\mathbb{Z}$-grading, where $G$ is a linearly ordered abelian group.
        
    \begin{proposition}\label{equiv}
        Let $G$ be a linearly ordered abelian group and let \[\Gamma: M_n(F)=R=\oplus_{g\in G}R_g\] be a grading on $M_n(F)$ such that the neutral component is commutative. Then the following statements are equivalent:
            
        \begin{enumerate}[(a)]
                \item $\Gamma$ is equivalent to the canonical $\mathbb{Z}$-grading of $M_n(F)$.
                \item The support of the grading has exactly $2n-1$ elements.
                \item There exists $g>0$ in $G$ such that $\Gamma$ is the elementary grading induced by $(0,g,2g,\dots, (n-1)g)$.
                \item There exists $g\in G$ such that $\dim R_g=n-1$.
        \end{enumerate}
    \end{proposition}
       
        \begin{proof}
            By Remark \ref{grad} we get $\Gamma$ is an elementary grading. The neutral component is commutative, hence $\Gamma$ is induced by an $n$-tuple of pairwise distinct elements of $G$. As a consequence of Remark \ref{isom} we may assume, without loss of generality, that $\Gamma$ is an elementary grading induced by an $n$-tuple $(g_1,\dots, g_n)\in G^{n}$ such that $0=g_1<g_2<\cdots<g_n$.
            Note that $\spo(\Gamma)$ has at least $2n-1$ elements, indeed the $2n-1$ elements $e_{11},e_{12},\dots,e_{1n},e_{21},\dots,e_{n1}$ have pairwise distinct degrees.
            
            \vspace{0,3cm}
            \noindent
            $(a)\Rightarrow (b)$ It is immediate since the cardinality of the support is invariant under equivalence of gradings. 
            
            \vspace{0,3cm}
            \noindent
            $(b)\Rightarrow (c)$ Now assume that $\spo( \Gamma)$ has $2n-1$ elements, then  \[\spo(\Gamma)=\{0,g_2,\dots,g_n,-g_2,\dots,-g_n\}.\] In this case, for each $i<j$, $g_j-g_i$ is one among $g_2,\dots,g_n$. By Lemma \ref{campinense}, for each $k$, $g_k=(k-1)g_2$. 
            
            \vspace{0,3cm}
            \noindent
            $(c)\Rightarrow (d)$
            Notice that the set $e_{12}, e_{23},\dots , e_{n-1,n}$ is a linear basis of the subspace $R_g$, hence $\dim R_g=n-1$.
            
            \vspace{0,3cm}
            \noindent
            $(d)\Rightarrow (c)$ Up to an isomorphism of gradings, by Remark \ref{isom}, we may assume the gradings is induced by the $n$-tuple $(g_1,\dots, g_n)$ with $0=g_1<g_2<\cdots < g_n$. Let $g\neq 0$ be an element of $G$ such that $\dim R_g=n-1$. Since $\dim R_g=\dim R_{-g}$, we may assume $g>0$. If $\deg e_{ij}>0$, then $i<n$, therefore there exist $i_1,\dots, i_{n-1}$ such that the elementary matrices with degree $g$ are $e_{1i_1},\dots, e_{n-1,i_{n-1}}$. For $1\leq k<l\leq n-1$ the equality $g_k-g_{i_k}=g_l-g_{i_l}$ implies that $g_{i_k}<g_{i_l}$. Therefore $1<i_1<\cdots<i_{n-1}\leq n$. Thus $i_{t}=t+1$ for $t=1,\dots, n-1$. Hence $g=g_2-g_1=g_3-g_2=\dots=g_n-g_{n-1}$ and for each $k\in\{1,\dots,n\}$, $g_k=(k-1)g_2$. 
            
            \vspace{0,3cm}
            \noindent
            $(c) \Rightarrow (a)$ It is immediate that the map induced by $1\in \mathbb{Z} \mapsto g\in G$ gives an equivalence between the canonical $\mathbb{Z}$-grading of $R$ and the grading $\Gamma$.
        \end{proof}
        
        Notice the hypothesis of $G$ being linearly ordered cannot be removed in Proposition \ref{equiv} as we show in the next example.
        
        \begin{remark}\label{Z2n-1}
               Let $R$ be endowed with the $\mathbb{Z}_{2n-1}$-grading induced by the $n$-tuple $(\overline{0}, \overline{1}, \dots, \overline{n-1})$. We have  $R_{\overline{n-1}}=\langle E_{1n}\rangle$. In particular, $R_{\overline{n-1}}^{2}=0$,  and this yields a non-trivial monomial identity, since $\overline{2(n-1)}=-\overline{1}$ and $R_{-\overline{1}}\neq 0$. The grading $R$ is equivalent to the canonical $\mathbb{Z}$-grading on $M_n(F)$, however it satisfies a non-trivial monomial identity. 
        \end{remark}
         
          The next result shows some characterizations of the canonical $\Z$-grading on $R$. This will be helpful in order to determine the almost non-degenerate gradings on $R$. 
        
        \begin{theorem}\label{almnond}
            If $G$ is a linearly ordered abelian group and $\Gamma$ is a $G$-grading on $R$ such that $|\spo(\Gamma)|=2n-1$, then $\Gamma$ is an almost non-degenerate grading.
        \end{theorem}
        \begin{proof}
        Proposition \ref{equiv} implies that there exists $h>0$ in $G$ such that $\Gamma$ is the elementary grading on $M_n(F)$ induced by $(0,h,2h,\dots, (n-1)h)$. Therefore $\spo(\Gamma)\subseteq H$, where $H$ is the subgroup of $G$ generated by $h$. Let $A$ denote $M_n(F)$ with the canonical $\mathbb{Z}$-grading. The map $\varphi: z\mapsto zh$ is an isomorphism from $\mathbb{Z}$ to $H$ such that $A_z=R_{\varphi(z)}$ for every $z\in\mathbb{Z}$. This implies that the gradings on $R$ and $A$, viewed as gradings by the groups $\mathbb{Z}$ and $H$ respectively, are weakly isomorphic. Proposition \ref{almnd} implies that the $H$-grading on $R$ is also almost non-degenerate. This gives that $\Gamma$ is an almost non-degenerate grading.
        \end{proof}
        
        A natural question now arises. Does the converse of Proposition \ref{almnond} hold? More precisely, if $G$ is a linearly ordered abelian group and the $G$-grading $\Gamma$ on $M_n(F)$ is almost non-degenerate does it follow that $\Gamma$ is equivalent to the canonical $\mathbb{Z}$-grading on $M_n(F)$?
		
		The answer to the above question is negative and counterexamples exist only for $n\geq 4$, as we can see in what follows.
		
		\begin{proposition}\label{counterexample}
    	    Let $G$ be a linearly ordered group and $\Gamma$ be an almost non-degenerate elementary $G$-grading on $R$, $n\leq 3$, such that the neutral component is commutative. Then $\Gamma$ is equivalent to the canonical $\mathbb Z$-grading on $R$.
		\end{proposition}
		
	    \begin{proof}
	        Let us consider $n=2$, then for any non-trivial $G$-grading there exists $g\in G$ such that $$R_0=\left(\begin{array}{cc}
	            F & 0 \\
	            0 & F
	        \end{array}\right), R_g=\left(\begin{array}{cc}
	            0 & F \\
	            0 & 0
	        \end{array}\right),  R_{-g}=\left(\begin{array}{cc}
	            0 & 0 \\
	            F & 0
	        \end{array}\right).$$ It is clear that this grading is equivalent to the canonical $\mathbb{Z}$-grading on $R$. If $n=3$, by Remark \ref{grad} and Remark \ref{isom} we may assume that $\Gamma$ is the elementary grading induced by the triple $\overline{g}=(0,g_1,g_2)$ of elements of $G$, where $0<g_1<g_2$. In this case the positive elements in $\spo(\Gamma)$ are $\{g_1, g_2, g_2-g_1\}$.
	        If $g_2\neq 2g_1$ we have $$R_{g_1}=\left(\begin{array}{ccc}
	            0 & F & 0\\
	            0 & 0 & 0\\
	            0 & 0 & 0
	        \end{array}\right), R_{g_2}=\left(\begin{array}{ccc}
	            0 & 0 & F\\
	            0 & 0 & 0\\
	            0 & 0 & 0
	        \end{array}\right), R_{g_2-g_1}=\left(\begin{array}{ccc}
	            0 & 0 & 0\\
	            0 & 0 & F\\
	            0 & 0 & 0
	        \end{array}\right).$$ As a consequence, it satisfies a non-trivial monomial identity of length 2: $x_1x_2=0$, if $\deg x_1=g_2-g_1$ and $\deg x_2=g_1$. This contradicts the fact that $\Gamma$ is almost non-generate. Therefore we get $g_2=2g_1$ which leads us to state the grading on $R$ is equivalent to the canonical $\mathbb{Z}$-grading on $R$.
	    \end{proof}
	    
	    We now introduce a grading on $R$ which will play an important role in the characterization of its almost non-degenerate gradings.
	    
    Let $n$ be a positive integer and denote by $e_i$ the element of $\mathbb{Z}^{\left \lfloor{\frac{n}{2}}\right \rfloor}$ whose $i$-th entry is equal to $1$ and the remaining entries are $0$. We consider the $(n-1)$-tuple $\overline{d}=(d_1,\dots, d_{n-1})$ where $d_i=e_i$ for $1\leq i \leq \left \lfloor{\frac{n}{2}}\right \rfloor$ and $d_i=d_{n-i}$ for $ \left \lfloor{\frac{n}{2}}\right \rfloor<i\leq n-1$. Now let $\overline{g}=(g_1,\dots,g_n)$ be the $n$-tuple of elements of $\mathbb{Z}^{\left \lfloor{\frac{n}{2}}\right \rfloor}$ such that $g_1=0$ and $g_{i+1}-g_i=d_i$ for $i=1,\dots, n$. Let  
\begin{align}\label{grading}
    R=\oplus_{g\in \mathbb{Z}^{\left \lfloor{\frac{n}{2}}\right \rfloor}}R_g
\end{align}
be the elementary $\mathbb{Z}^{\left \lfloor{\frac{n}{2}}\right \rfloor}$-grading on $R$ induced by $\overline{g}$. Then we get the following.

\begin{proposition}\label{coars}
	Let $G$ be an abelian group without elements of order $2$. Let us consider $R$ with an elementary grading induced by a tuple $(g_1,\dots, g_n)$ of pairwise distinct elements of $G$. If there exists $g\in G$ such that $\dim R_g=1$ and $R$ satisfies no non-trivial multilinear identity of length $2$, then this grading is isomorphic to a coarsening of the $\mathbb{Z}^{\left \lfloor{\frac{n}{2}}\right \rfloor}$-grading in (\ref{grading}).
\end{proposition}

\begin{proof}
	Let $R_g:=span\{e_{ij} \}$ be a one dimensional component of the grading and let $\overline{h}=(h_1,\dots, h_n)$, where $h_t=g_t-g_i$ for $t=1,\dots, n$. 
	
	We claim that for every $t\in\{1,\dots, n\}$ there exists an index $l(t)$ such that $e_{l(t)j}\in R_{h_t}$.  Otherwise there exists $t_0$ such that $e_{lj}\notin R_{h_{t_0}}$ for $l=1,\dots, n$ or, equivalently, $e_{jl}\notin R_{-h_{t_0}}$, for $l=1, \dots, n$. This implies that $m=x_1x_2$, with $\deg x_1=g$ and $\deg x_2=-h_{t_0}$, is a graded monomial identity for $R$. Since $\deg m=g-h_{t_0}=g_j-g_{t_0}\in\spo(R)$, then $m$ is a non-trivial monomial identity, that is a contradiction. Thus we obtain a map $t\mapsto l(t)$, where $l(t)$ is the index such that $h_j-h_{l(t)}=g_j-g_{l(t)}=h_t$. This map is injective and therefore a permutation in $S_n$. Hence we conclude that $h_j-h_t\in \{h_1,\dots, h_n\}$, for $t\in \{1,\dots, n\}$.
	Thus we have a permutation $\alpha\in S_n$ so that 
	\begin{equation}\label{alpha}
	h_j-h_t=h_{\alpha(t)},
	\end{equation}
	for $t=1,\dots,n$. Since the group $G$ is abelian we conclude that $\alpha^2=1$, hence $\alpha$ is the product of disjoint transpositions. Since $G$ has no elements of order $2$, the homomorphism $g\mapsto 2g$ is injective and $\alpha$ has at most one fixed point. We may consider, up to isomorphism, $R$ being graded by the tuple $(h_{\beta(1)},\dots, h_{\beta(n)})$, for some $\beta \in S_n$. The effect of this in (\ref{alpha}) is to change $\alpha$ for $\beta^{-1}\alpha \beta$. Indeed, let $h_t'=h_{\beta(t)}$, for $t=1,\dots, n$ and let us rewrite (\ref{alpha}) for the tuple $(h_1', \dots, h_n')$. Let $j'=\beta^{-1}(j)$, then 
	\[h'_{j'}-h'_{t}=h_j-h_{\beta(t)}=h_{\alpha\beta(t)}=h'_{\beta^{-1}\alpha\beta(t)}.\]
	Since $\alpha$ has at most one fixed point, we can choose, without loss of generality, $\alpha$ being of the following form
	\[\alpha=\left(\begin{array}{ccccc}
	1&2&\dots&n-1&n\\
	n&n-1&\dots&2&1
	\end{array}\right).\] In this case, (\ref{alpha}) implies that $h_t+h_{n+1-t}=h_j$ for $t=1,\dots, n$. Now let $d_t=h_{t+1}-h_t$ for $t=1,\dots, n-1$. As a consequence we have
	\[d_{n-t}-d_t=h_{n-t+1}-h_{n-t}-(h_{t+1}-h_t)=h_{n-t+1}+h_{t}-(h_{n-t}+h_{t+1})=0,\]
	and $d_t=d_{n-t}$, which means the grading is a coarsening of the grading in (\ref{grading}).
\end{proof}

\begin{remark}
The converse of the above proposition does not hold and the grading in Remark \ref{Z2n-1} provides a counter-example since it is isomorphic to a coarsening of the grading in (\ref{grading}), has a 1-dimensional component but satisfies a monomial identities of length 2, for $n>2$.
\end{remark}

Looking back at the hypothesis of Proposition \ref{coars}, it is natural to ask whether or not the fact that $R$ satisfies a non-trivial graded monomial identity, implies $R$ satisfies a non-trivial monomial identity of length 2. The answer to this question is negative and Example \ref{exam} furnishes a counterexample. In fact, the graded algebra in Example \ref{exam} satisfies a non-trivial monomial identity of length $n$ but does not satisfies any non-trivial monomial of length 2. 

\begin{remark}\label{coarse2}
Let us specialize Proposition \ref{coars} for $G$ being a linearly ordered abelian group and $R=M_n(F)$ with a grading induced by an $n$-tuple $(g_1,\dots, g_n)$ of pairwise distinct elements of $G$. Up to isomorphism of gradings, we may consider $0=g_1<g_2<\dots<g_n$. In particular, $R_{g_n}$ plays the role of the 1-dimensional component of $R$ and $R_{g_n}=span\{e_{1n}\}$. In the notation of the proof of Proposition \ref{coars}, $i=1$, $j=n$ and  since $g_n-g_1>g_n-g_2>\cdots>g_n-g_n$, and all of them lie in $\{g_1,\dots, g_t\}$, we have $g_n-g_t=g_{n-t+1}$. In this case, this grading is a coarsening of the grading in (\ref{grading}), defined by an  $n-1$-tuple $(d_1,\dots,d_{n-1})$, with $d_t>0$, for all $t$. More precisely $d_t=g_{t+1}-g_t$.

\end{remark}

Now we are ready to state our first main result of the section linking graded identities of $M_4(F)$ and almost non-degenerate gradings.

\begin{theorem}\label{n=4}
    Let $G$ be a linearly ordered abelian group. Let $R$ denote the algebra $M_4(F)$ with a $G$-grading such that the neutral component is commutative. Then the $T_G$-ideal $T_G(R)$ is generated as a $T_G$-ideal by the identities $(\ref{(3)})-(\ref{(5)})$ if and only if there exist $a,b>0$ in $G$, with $a\neq 2b$, $2a\neq b$, such that $R$ is isomorphic to the elementary grading on $M_4(F)$ induced by $(0,a,a+b,2a+b)$.
\end{theorem}
	    
\begin{proof}
    We assume that $T_G(R)$ is generated by the identities $(\ref{(3)})-(\ref{(5)})$.  Proposition \ref{coars} and Remark \ref{coarse2} imply that there exist $a,b>0$ in $G$ such that $R$ is isomorphic to the elementary grading induced by $(0, a, a+b, 2a+b)$. If $a=b$, then by Theorem \ref{almnond} we have $T_G(R)$ is generated by $(\ref{(3)})-(\ref{(5)})$. Now assume that $a\neq b$. If $a=2b$ or $2a=b$, then $R_{2b}\neq 0$ and $R_{b}=\langle e_{23}\rangle$, hence $x_{1}x_{2}$, with $\deg x_1 =\deg x_2=b$, is a non-trivial monomial identity, which is a contradiction. Therefore we get $a\neq 2b$ and $2a\neq b$.
    
    Now assume $R$ is endowed with the elementary grading induced by $(0,a,a+b,2a+b)$ with $a$, $b>0$, $a\neq b$, $a\neq 2b$ and $2a\neq b$. In this case the non-zero homogeneous components of $R$ are 
    \begin{align*}
    &R_0=\langle e_{11}, e_{22}, e_{33}, e_{44} \rangle, R_a=\langle e_{12},e_{34}\rangle, R_b=\langle e_{23}\rangle, \\&R_{a+b}=\langle e_{13}, e_{24} \rangle, R_{2a+b}=\langle e_{14} \rangle,R_{-a}=\langle e_{21}, e_{43}\rangle,\\ &R_{-b}=\langle e_{32}\rangle, R_{-a-b}=\langle e_{31}, e_{42} \rangle, R_{-2a-b}=\langle e_{41} \rangle.
    \end{align*}
    By Theorem \ref{monomialdegree}, we may consider only monomials of length up to $4$. Let us assume first that there exists a monomial $M=x_{1}x_{2}x_{3}x_{4}$ of length $4$, with $\deg x_i=g_i\in G$, that is not a consequence of monomials of length less than 4. Note that if $R_{g_i}R_{g_{i+1}}= R_{g_i+g_{i+1}}$ for some $i\in\{1,2,3\}$, then the monomial $N$ of degree $3$ obtained by replacing $x_{i}x_{i+1}$ in $M$ by a variable of degree $g_i+g_{i+1}$ is an identity for $R$. It is clear that $M$ is a consequence of $N$, which is a contradiction.
	If $M$ has a submonomial of degree 0, the same argument in the proof of Proposition \ref{monomialdegree} implies that $M$ is a consequence of a monomial identity of length $3$, and we may assume that $M$ does not have a submonomial of degree zero. In particular, each variable has non-zero degree.
	Direct verification shows that $\dim (R_gR_h)=1$ whenever $0\neq g,h \in \spo(R)$ and $0\neq g+h\in \spo(R)$. Hence we get $\dim R_{g_i}R_{g_{i+1}}\leq 1$ for $i=1,2,3$. If $R_{g_i}R_{g_{i+1}}=0$ for some $i\in\{1,2,3\}$, then $M$ is a consequence of a monomial identity of degree 2, which is a contradiction. Therefore we have $\dim R_{g_i}R_{g_{i+1}}= 1$ for $i=1,2,3$. The monomials $x_{1}x_{2}x_{3}$, $x_{2}x_{3}x_{4}$ are not identities for $R$ and, as a consequence, 
	$\dim R_{g_1}R_{g_2}R_{g_3}=
	\dim R_{g_2}R_{g_3}R_{g_4}=1$. This implies that there exist uniquely determined tuples $(i_1,i_2,i_3, i_4)$ and $(i_2^{\prime}, i_3^{\prime}, i_4^{\prime},i_5)$ such that $e_{i_1i_2}, e_{i_2i_3}, e_{i_3i_4}$ have degree $g_1,g_2,g_3$, respectively and  $e_{i_2^{\prime}i_3^{\prime}}, e_{i_3^{\prime}i_4^{\prime}}, e_{i_4^{\prime}i_5}$ have degree $g_2,g_3,g_4$, respectively. Note that $e_{i_2i_4}=e_{i_2i_3}e_{i_3i_4}$ and $e_{i_2^{\prime}i_4^{\prime}}=e_{i_2^{\prime}i_3^{\prime}}e_{i_3^{\prime}i_4^{\prime}}$ lie in $R_{g_2}R_{g_3}$. Since $\dim(R_{g_2}R_{g_3})=1$ we obtain $i_2=i_2^{\prime}$, $i_3=i_3^{\prime}$ and $i_4=i_4^{\prime}$. It means the matrices $e_{i_1i_2},  e_{i_2i_3}, e_{i_3i_4}, e_{i_4i_5}$ have degree $g_1,g_2,g_3,g_4$, respectively. Then $e_{i_1i_2}  e_{i_2i_3} e_{i_3i_4} e_{i_4i_5}=e_{i_1 i_5}\neq 0$, which is a contradiction, since $M$ is a graded identity for $R$. Hence, every non-trivial identity of length 4 is a consequence of identities of shorter length.
	
	We suppose now $m=x_1x_2x_3$ is a monomial identity of degree $3$, with $\deg x_i=g_i$, for $i=1,2,3$. First we observe that if $\dim R_{g_2}=1$, then $m$ is a consequence of one of the monomial identities $x_{g_1}x_{g_2}$ or $x_{g_2}x_{g_3}$, or $\dim R_{g_1}R_{g_2}=\dim R_{g_2}R_{g_3}=1$ and, in this case, $m$ is not an identity. Hence, it remains to consider the case $\dim R_{g_2}=2$. We must have $\dim R_{g_3}=2$ or $\dim R_{g_1}=2$, otherwise, by direct verification, $m$ is a consequence of a monomial of length 2. In the cases above, we may verify that $R_{g_2}R_{g_3}=R_{g_2+g_3}$ or $R_{g_1}R_{g_2}=R_{g_1+g_2}$. In any case, $m$ is a consequence of a monomial identity of length 2.

	Finally we consider the case of monomial identities of degree 2. In this case, we just need to observe that they are consequences of  $x_g=0$, with $g\not\in Supp(\Gamma)$. One can check that, since $a, b>0$, this happens if and only if $a\neq b$, $a\neq 2b$ and $b\neq 2a$.
	\end{proof}
		
		Now we are going to show an analogous result about $M_5(F)$.
		
\begin{theorem}\label{n=5}
   Let $G$ be a linearly ordered abelian group. Let $R$ denote the algebra $M_5(F)$ with a $G$-grading such that the neutral component is commutative. The $T_G$-ideal $T_G(R)$ is generated, as a $T_G$-ideal, by the identities $(\ref{(3)})-(\ref{(5)})$ if and only if there exist $a,b>0$ in $G$, with $a\neq 2b$, $b\neq 2a$, $a\neq 3b$ and $a\neq 4b$, such that $R$ is isomorphic to the elementary grading on $M_5(F)$ induced by $\overline{g}=(0,a,a+b,a+2b,2a+2b)$.
\end{theorem}
		
\begin{proof}
   We assume that $T_G(R)$ is generated by the identities $(\ref{(3)})-(\ref{(5)})$.  Proposition \ref{coars} and Remark \ref{coarse2} imply that there exist $a,b>0$ in $G$ such that $R$ is isomorphic to the elementary grading induced by $(0,a,a+b,a+2b,2a+2b)$. Now we prove that $a\neq 2b$, $b\neq 2a$, $a\neq 3b$ and $a\neq 4b$. If $a=b$ then these inequalities are clearly true, now assume that $a\neq b$.
   If $a=2b$, then $2a=a+2b\in \spo(R)$ and $x_{1}x_{2}$, with $\deg x_1=\deg x_2=a$, is a non-trivial graded monomial identity.
   If $2a=b$, then $2a=b\in \spo(R)$ and $x_{1}x_{2}$, with $\deg x_1=\deg x_2=a$, is a non-trivial graded monomial identity.
   If $a=3b$, then $a-b=2b\in \spo(R)$ and $x_1x_2=0$, with $\deg(x_1)=a$ and $\deg(x_2)=-b$, is a non-trivial graded monomial identity again.
   Finally, if $a=4b$, then $4b\in \spo(R)$ and $x_1x_2=0$, with $\deg(x_1)=2b$ and $\deg(x_2)=2b$, is a non-trivial graded monomial identity.
   Since by assumption, $R$ is generated by $(\ref{(3)})-(\ref{(5)})$, it does not satisfy non-trivial graded monomial identities and this implies that  $a\neq 2b$, $b\neq 2a$, $a\neq 3b$ and $a\neq 4b$.
    
    Now assume the elementary grading on $R$ is induced by $(0,a,a+b,a+2b,2a+2b)$ with $a, b>0$ in $G$ and $a\neq 2b$. If $a=b$, then Theorem \ref{almnond} implies that $T_G(R)$ is generated by $(\ref{(3)})-(\ref{(5)})$. Now assume $a\neq b$. In this case the non-zero homogeneous components of $R$ are 
    \begin{align*}
    &R_0=\langle e_{11}, e_{22}, e_{33}, e_{44}, e_{55} \rangle, R_{a}=\langle e_{12}, e_{45} \rangle, R_{a+b}=\langle e_{13}, e_{35} \rangle,\\& R_{a+2b}=\langle e_{14}, e_{25} \rangle, R_{2a+2b}=\langle e_{15} \rangle, R_{b}=\langle e_{23}, e_{34}\rangle, R_{2b}=\langle e_{24} \rangle\\& R_{-a}=\langle e_{21}, e_{54} \rangle, R_{-a-b}=\langle e_{31}, e_{53} \rangle, R_{-a-2b}=\langle e_{41}, e_{52} \rangle,\\& R_{-2a-2b}=\langle e_{51}  \rangle, R_{-b}=\langle e_{32}, e_{43}\rangle, R_{-2b}=\langle e_{42}\rangle
    \end{align*}
   Direct calculations show that $\dim (R_gR_h)=1$ whenever $0\neq g,h \in \spo(R)$ and $0\neq g+h\in \spo(R)$. 
   Again, by Theorem \ref{monomialdegree}, we may consider only monomials of length up to 5. Following the same arguments in the proof of Theorem \ref{n=4}, we obtain any monomial identity of length 3, 4, or 5 is a consequence of a monomial identity of shorter length. 
   
	Finally, we consider the case of monomial identities of degree 2. By direct computations we obtain that, since $a, b>0$, the graded monomial identities of length 2 are consequences of $x_g=0$, with $g\not\in Supp(\Gamma)$ if $a\neq 2b$ and $b\neq 2a$, $a\neq 3b$ and $a\neq 4b$ and we are done.
	\end{proof}

	The problem of the existence of non-trivial monomial identities has the following combinatorial interpretation. Let $(g_1,\dots, g_n)$ be a tuple of integers and consider the set \[S:=\{g_i-g_j, 1\leq i,j\leq n\}.\] For any $k\in S$ let $X_k$ be the matrix with $1$ in the positions $(i,j)$ such that $g_j-g_i=k$ and $0$ elsewhere. We consider the complete directed graph $G$ with vertices $\{v_1,\dots, v_n\}$  (i.e. a graph with arrows $v_i\rightarrow v_j$ for every $1\leq i,j \leq n$) and arrows $v_i\rightarrow v_j$ labeled by $g_i-g_j$. Then for $k_1,\dots, k_t\in S$, the $(i,j)$-th entry of $X_{k_1}\cdots X_{k_t}$ is the number of paths from $v_i$ to $v_j$ going through $t$ vertices (not necessarily distinct) such that the sequence of numbers in the arrows of the path is $(k_1,\dots, k_t)$. We now consider the following "elementary" related problem. We say $(g_1,\dots, g_n)$ is a \textit{good sequence} of integers if $X_{k_1}\cdots X_{k_t}=0$ implies that there exist $1\leq i\leq j \leq t$ such that $k_i+\cdots +k_j\notin S$ or one of the products $X_{k_1}\cdots X_{k_{t-1}}$, $X_{k_2}\cdots X_{k_t}$ is zero. 
	\begin{problem}
		Determine all good sequences of integers.
	\end{problem}


\begin{thebibliography}{99}
\bibitem{AB} \textrm{E. Aljadeff, A. Kanel-Belov}, \textit{Representability and Specht problem for $G$-graded algebras}, Adv. Math. {\bf 225}(5) (2010), 2391--2428.
\bibitem{AO} \textrm{E. Aljadeff, D. Ofir}, \textit{On group gradings on PI-algebras}, J. Algebra \textbf{428} (2015), 403--424.
\bibitem{BD} \textrm{Yu. Bahturin, V. Drensky}, \textit{Graded polynomial identities of matrices}, Linear Algebra Appl. {\bf 357} (2002), 15--34.
\bibitem{bgr1} \textrm{Yu. Bahturin, A. Giambruno, D. Riley}, \textit{Group-graded algebras with polynomial identities}, Israel J. Math. \textbf{104} (1998), 145--155.
\bibitem{BZ1}  \textrm{Yu. A. Bahturin, M. V. Zaicev}, \textit{Group gradings on matrix algebras. Dedicated to Robert V. Moody}, Canad. Math. Bull. \textbf{45}(4) (2002), 499--508.
\bibitem{CM}\textrm{L. Centrone, T. C. de Mello}, \textit{On $\mathbb{Z}_n$-graded identities of block-triangular matrices}, Linear Multilinear Algebra \textbf{63} (2015), 302--313.

\bibitem{Das} \textrm{S. D\u{a}sc\u{a}lescu, B. and Ion, C. N\u{a}st\u{a}sescu, and J. Rios Montes}, \textit{Group gradings on full matrix rings}, J. Algebra \textbf{220} (2) (1999), 709--728.

\bibitem{DM}\textrm{D. Diniz, T.C. de Mello}, \textit{Graded identities of block-triangular matrices}, J. Algebra \textbf{464} (2016), 246--265.
\bibitem{DN}\textrm{O. M. Di Vincenzo, V. Nardozza},
\textit{Graded polynomial identities of verbally prime algebras}, J. Algebra Appl. {\bf 6}(3) (2007), 385--401.
\bibitem{D1} \textrm{V. Drensky}, \textit{Identities of representations of nilpotent Lie algebras}, Comm. Algebra \textbf{25} (7) (1997), 2115--2127.
\bibitem {D} \textrm{V. Drensky}, \textit{A minimal basis for a second-order matrix algebra over a field of characteristic 0}, Algebra i Logika \textbf{20} (3) (1980), 282--290 [in Russian]; Algebra and Logic \textbf{20} (3) (1981), 188--194 [Engl. transl.].
\bibitem{EK} A. Elduque, M. Kochetov, \textit{Gradings on simple Lie algebras}, Mathematical Surveys and Monographs, {\bf 189}. American Mathematical Society, Providence, RI; Atlantic Association forResearch in the Mathematical Sciences (AARMS), Halifax, NS, 2013.
\bibitem {DV} \textrm{O. M. Di Vincenzo}, \textit{On the graded identities of $M_{1,1}(E)$}, Israel J. Math. \textbf{80} (1992), 323--335.
\bibitem{I}\textrm{I. Sviridova}, \textit{Identities of pi-algebras graded by a finite abelian group}, Comm. Algebra {\bf 39}(9) (2011), 3462--3490.
\bibitem {K} \textrm{A. Kemer}, \textit{Solution of the problem as to whether associative algebras have a finite basis of identities}, Dokl. Akad. Nauk SSSR \textbf{298} (1988), 273--277 [in Russian]; Soviet Math. Dokl. \textbf{37} (1988), 60--64 [Engl. transl.].
\bibitem {K2} \textrm{A. Kemer}, \textit{Ideals of Identities of Associative Algebras}, Translations of Monographs 87, Amer. Math. Soc., Providence, RI, 1991.
\bibitem{Ko2}  \textrm{P. Koshlukov}, \textit{Basis of the identities of the matrix algebra of order two over a field of characteristic $p>2$}, J. Algebra \textbf{241}(1) (2001), 410--434.
\bibitem{Ko1} \textrm{P. Koshlukov}, \textit{Ideals of identities of representations of nilpotent Lie algebras}, Comm. Algebra \textbf{28}(7) (2000), 3095--3113.
\bibitem {R} \textrm{Yu. P. Razmyslov}, \textit{The existence of a finite basis for the identities of the matrix algebra of order two over a field of characteristic zero}, Algebra i Logika \textbf{12} (1) (1973), 83--113 [in Russian].
\bibitem {V} \textrm{S. Yu. Vasilovsky} \textit{$\mathbb{Z}$-graded polynomial identities of the full matrix algebra} Comm. Algebra \textbf{26} (1998), 601--612.
\bibitem {V2} \textrm{S. Yu. Vasilovsky} \textit{$\mathbb{Z}_n$-graded polynomial identities of the full matrix algebra of order $n$} Proc. Amer. Math. Soc. \textbf{127} (1999), 3517--3524.



\end{thebibliography}
\end{document}